\pdfoutput=1

\documentclass[11pt,a4paper]{article}

\usepackage[english]{babel}
\usepackage[numbers,sort&compress]{natbib}
\usepackage[a4paper,top=2cm,bottom=2cm,left=2.6cm,right=2.8cm,marginparwidth=1.75cm]{geometry}

\usepackage[utf8]{inputenc}
\usepackage[T1]{fontenc}

\usepackage{amsmath,amssymb,amsfonts,amsthm,mathtools}
\usepackage{bm}
\usepackage{bbm}
\usepackage{esvect}
\allowdisplaybreaks

\usepackage{graphicx}
\usepackage{booktabs}
\usepackage{adjustbox}
\usepackage{caption}
\usepackage{subcaption}
\usepackage{float}
\usepackage{placeins}

\usepackage{algorithm}
\usepackage{algpseudocode}

\usepackage{tikz}
\usetikzlibrary{decorations.pathreplacing,decorations.markings}
\usepackage{pifont}
\usepackage{enumitem}
\usepackage{verbatim}
\usepackage{authblk}

\usepackage{xcolor}
\usepackage{titlesec}

\definecolor{mycitecolor}{RGB}{70,10,200}
\usepackage[colorlinks=true,linkcolor=blue,citecolor=mycitecolor,urlcolor=black]{hyperref}

\usepackage[nameinlink,noabbrev]{cleveref}

\theoremstyle{myplain}

\newtheorem{theorem}{Theorem}[section]
\newtheorem{lemma}[theorem]{Lemma}
\newtheorem{proposition}[theorem]{Proposition}
\newtheorem{corollary}[theorem]{Corollary}
\newtheorem{conjecture}[theorem]{Conjecture}

\theoremstyle{remark}
\newtheorem{remark}[theorem]{Remark}

\newcommand{\blue}{\mathrm{blue}}
\newcommand{\floor}[1]{\left\lfloor #1\right\rfloor}
\newcommand{\ceil}[1]{\left\lceil #1\right\rceil}

\title{A Sharp Ramsey Theorem for Admissible Colorings\\
of Ordered Cliques}
\author[1]{Yanan Hu \thanks{E-mail: hg@sit.edu.cn}}
\author[2]{Zhenhua Lyu\thanks{Corresponding author. E-mail: lyuzhh@outlook.com
}}
\author[3]{Chenxi Yang\thanks{E-mail: yangchenxi2022@email.szu.edu.cn }}

\affil[1]{\small School of Science, Shanghai Institute of Technology, Shanghai 201418, China}
\affil[2]{School of Science, Shenyang Aerospace University,
Shenyang, 110136, China}
\affil[3]{School of Mathematical Sciences, Shenzhen University,
Shenzhen, 518060, China}
\date{}

\begin{document}
\maketitle
\begin{abstract}
Let \(f(k)\) be the minimum integer \(N\) such that any red--blue edge-coloring of the ordered complete graph on \(N\) vertices contains a set of \(k\) vertices  whose induced coloring is admissible. In this note, we obtain the exact value of $f(k)$ for $k\ge 3$, which confirms a conjecture posed by  Brada\v{c}, Liu, Wu and Xu.
\end{abstract}



\section{Introduction}
In \cite{be}, Bollob\'as, Erd\H os and Simonovits obtained a stronger version of the  renowned  Erd\H os-Stone-Simonovits theorem.  Nikiforov \cite{niki} enhanced the above result by replacing the  edge density condition by a weaker clique density condition. More precisely, by the so-called supersaturation result of Erd\H os and Simonovits \cite{es}, any graph with $\varrho_2(G)>1-\frac{1}{r-1}$ not only contains a single $r$-clique, but has positive $K_r$-density $\varrho_r(G)>0$. Nikiforov \cite{niki} showed that a graph having $\varrho_r(G)>0$ already suffices to imply a logarithmic blowup of $K_r$. Brada\v{c}, Liu, Wu and Xu \cite{BradacLiuWuXu} proved that, in several structured graph classes, a fixed positive clique density guarantees not only the logarithmic-size blow-ups of the general Nikiforov theorem, but also complete multipartite blow-ups of size $n/log n$, or even linear size. In particular, the ordered-graph setting exhibits a striking quadratic threshold phenomenon.
In the ordered-graph case, a key ingredient is the auxiliary finite Ramsey-type parameter $f(k)$, for which they proved a quadratic upper bound.

We follow the notation of \cite{BradacLiuWuXu}. In this paper, consecutive vertices always refer to the order induced on the selected vertex set.  The dependency digraph may contain arcs in both directions between two vertices; such a pair of opposite arcs is regarded as a directed cycle of length two.  An \emph{ordered graph} $G_<$ is a graph with a total ordering  $<$ on its vertex set \(V(G)\). For a positive integer \(m\), denote \([m]=\{1,\ldots,m\}\). 
Let $\chi$ be a red-blue edge-coloring of the ordered clique with vertices $v_1<$ $v_2<\ldots<v_N$. Its \emph{dependency digraph} $D=D(\chi)$ is defined on the vertex set $\left\{v_1, \ldots, v_N\right\}$ as follows. For every $i \in[N-1]$ and $j \in[N] \backslash\{i, i+1\}$, if $\chi\left(v_i v_{i+1}\right)=$ red and $\chi\left(v_i v_j\right)=\chi\left(v_{i+1} v_j\right)=$ blue, then $D$ contains the directed edges $\left(v_j, v_i\right)$ and $\left(v_j, v_{i+1}\right)$. We say that $\chi$ is admissible if $D(\chi)$ is acyclic.

For \(k\geq1\),  denote by \(f(k)\)  the minimum integer \(N\)  such that in any red-blue edge-coloring of the
ordered clique on \(N\)  vertices, there exist \(k\) vertices such that the coloring induced on these $k$ vertices is admissible. As noted above, this parameter plays a pivotal role in \cite{BradacLiuWuXu} and is also of independent interest.

In \cite{BradacLiuWuXu}, Brada\v{c} et al.\ formulated the following conjecture.
\begin{conjecture}[\cite{BradacLiuWuXu}]
\label{conj:C}
Let \(k\geq1\) be an integer. Then we have
\[
  f(k)=\left\lfloor\frac{k^2}{4}\right\rfloor+1.
\]
\end{conjecture}

They reported a computer verification of
Conjecture~\ref{conj:C} for \(k\leq6\).
The results of Brada\v{c} et al.\ also imply the following bounds for \(f(k)\).
\begin{theorem}[\cite{BradacLiuWuXu}]
\label{thm:previous}
Let \(k\geq3\) be an integer. Then we have
\[
  \left\lfloor\frac{k^2}{4}\right\rfloor
  < f(k)
  \leq\frac{k^2-k+2}{2}.
\]
\end{theorem} In this paper, we prove Conjecture~\ref{conj:C} for general $k$. More precisely, our main result is the following.

\begin{theorem}
\label{thm:main}
Let \(k\geq3\) be an integer. Then we have
\[
  f(k)=\left\lfloor\frac{k^2}{4}\right\rfloor+1.
\]
\end{theorem}
The lower bound is implicit in Theorem \ref{thm:previous}. Nevertheless, we provide a short direct proof via a balanced interval coloring.

The remainder of the paper is organized as follows.  In Section~2, we collect the necessary preliminary results, including the basic properties of the longest path layers and the endpoint deletion lemma.  In Section~3, we prove the main theorem. We establish the asymmetric path layer splicing lemma, from which the matching upper bound follows.

\section{Preliminaries}
For a graph $G$, we write $V(G)$ for its vertex set.  If $S\subseteq V(G)$, then $G[S]$ denotes the subgraph of $G$ induced by $S$.
Throughout this section, let \(
  V=\{v_1<v_2<\cdots<v_N\}\)
be the vertex set of an ordered complete graph, and let
\( \chi\)
be an arbitrary red--blue edge-coloring. An \emph{increasing blue path of order \(r\)} is a sequence
\[
  u_1<u_2<\cdots<u_r
\]
such that \(\chi(u_iu_{i+1})=\blue\) for every \(i\in[r-1]\).  For a vertex
\(v\), let \(\sigma(v)\) denote the maximum order of an increasing blue path
starting at \(v\).  Put
\[
  L=\max_v\sigma(v),
  \qquad
  S_t=\{v:\sigma(v)=t\}
  \quad (1\leq t\leq L).
\]

\begin{lemma}\label{lem:layers}
Let \(\chi\) be an arbitrary red--blue edge-coloring of the ordered
complete graph \(G\) on \(V\), and let \(L\) and \(S_1,\ldots,S_L\) be defined
as above.  Then the following statements hold.
\begin{enumerate}[label=\textup{(\roman*)}]
  \item The sets \(S_t\) form a partition of \(V\), and each \(S_t\) induces a complete graph without a blue edge.

 \item For all \(i\in[L]\), if
  \(
    P=(p_1<p_2<\cdots<p_L)
  \)
  is an increasing blue path of maximum order, then
  \[
    \sigma(p_i)=L-i+1.
  \]
  Consequently, the path \(P\) meets \(S_t\) in
  exactly one vertex.
\end{enumerate}
\end{lemma}

\begin{proof}
We first prove part~\textup{(i)}. The sets \(S_1,\ldots,S_L\) are pairwise disjoint and their union is \(V\), so they
form a partition of the vertex set.
It remains to show that every edge of $G_\chi[S_t]$ is red. 
The assertion is immediate if \(|S_t|\leq1\).  Otherwise, let \(u<v\) be
arbitrary distinct vertices of \(S_t\).  Then
\(
  \sigma(u)=\sigma(v)=t.
\)
There exists an increasing blue path $P$ of order \(t\) starting at \(v\).
Suppose, for a contradiction, that \(uv\) is blue.  Since \(u<v\),
the sequence
\(  uP\)
is an increasing blue path of order \(t+1\) starting at \(u\).  It follows that
\[
  \sigma(u)\geq t+1,
\]
a contradiction.  Therefore \(uv\) is red.

We next prove part~\textup{(ii)}.   For all \(i\in[L]\), the
subpath
\[
  p_i<p_{i+1}<\cdots<p_L
\]
is an increasing blue path of order \(L-i+1\) starting at \(p_i\).
Therefore
\(
  \sigma(p_i)\geq L-i+1.
\)

Suppose that the inequality is strict for some \(i\in[L]\).  By the
definition of \(\sigma(p_i)\), there exists an increasing blue path
\[
  p_i=q_1<q_2<\cdots<q_s
\]
of order
\(
  s=\sigma(p_i)>L-i+1.
\)
Then we obtain an increasing blue path
\[
  p_1<p_2<\cdots<p_{i-1}<p_i=q_1<q_2<\cdots<q_s.
\]
 The order is
\(
  (i-1)+s
  >(i-1)+(L-i+1)
  =L,
\)
a contradiction.  Thus we have 
\[
  \sigma(p_i)=L-i+1.
\]

For each \(t\in[L]\), the unique vertex of \(P\) having \(\sigma\)-value \(t\) is
\(p_{L-t+1}\).  Therefore
\[
  V(P)\cap S_t=\{p_{L-t+1}\}.
\]
\end{proof}

\begin{lemma}\label{lem:trimming}
Let \(\chi\) be a red--blue edge-coloring of an ordered complete graph
\(G\), and let
\[
  X=\{x_1<x_2<\cdots<x_m\}\subseteq V(G).
\]
Suppose that the coloring induced by \(\chi\) on \(X\) is admissible. If \(m\geq2\), then the coloring induced by \(\chi\) on
\(X\setminus\{x_1\}\) and \(X\setminus\{x_m\}\) are both admissible. Consequently, for all \(r\) with \(1\leq r\leq m\), there exists a subset \(Y\subseteq X\) of order exactly \(r\) such that the coloring induced by \(\chi\) on \(Y\) is admissible.
\end{lemma}

\begin{proof}

Assume that \(m\geq2\), let \(X^-=X\setminus\{x_1\}\).  Every pair of consecutive vertices in
\(X^-\) was already consecutive in \(X\).  Hence every dependency arc
of \(D_\chi(X^-)\) is also an arc of \(D_\chi(X)\), and therefore
\[
  D_\chi(X^-)\subseteq D_\chi(X)[X^-].
\]
Since \(D_\chi(X)\) is acyclic, so is \(D_\chi(X^-)\). It follows that the coloring induced by \(\chi\) on \(X^-\) is admissible.
The same argument, applied to
\(
  X^+:=X\setminus\{x_m\},
\)
shows that the coloring induced by \(\chi\) on \(X^+\) is admissible.

Finally, if \(r=m\), take \(Y=X\).  If \(r<m\), repeatedly delete the
leftmost or rightmost vertex of the current ordered set.  At every
step, the coloring induced on the remaining set is admissible by the
preceding argument.  After \(m-r\) deletions, the resulting subset
\(Y\subseteq X\) has order exactly \(r\), and the coloring induced by
\(\chi\) on \(Y\) is admissible.
\end{proof}

The lower bound below is implicit in the work of Brada\v{c} et al.\cite{BradacLiuWuXu}. It follows by combining \(g(k)\leq f(k)\) with their geometric lower-bound
construction for \(g(k)\), a clique-density threshold associated with forcing almost linear-sized bicliques in complements of ordered graphs with no induced monotone \(P_{2k}\).  For completeness, and to keep the present argument independent of that geometric framework, we give a short direct proof.

\begin{proposition}[\cite{BradacLiuWuXu}]\label{prop:lower}
Let \(k\geq 1\) be an integer. Then we have
\[
  f(k)\geq\floor{\frac{k^2}{4}}+1.
\]
\end{proposition}

\begin{proof}
First observe that \(f(k)\geq k\) for every \(k\geq1\), since an ordered complete graph on fewer than \(k\) vertices contains no
\(k\)-vertex set.  For \(1\leq k\leq3\), we have
\[
  \floor{\frac{k^2}{4}}+1=k,
\]
and hence the desired inequality follows.  Next we assume that \(k\geq 4\), and let
\[
  a=\floor{\frac{k}{2}},
  \qquad
  b=\ceil{\frac{k}{2}}.
\]
Partition \(ab=\floor{k^2/4}\) ordered vertices into \(a\) consecutive
blocks \(B_1<\cdots<B_a\), each of order \(b\).
Here, \(B_i<B_j\) means that whenever \(i<j\), every vertex of \(B_i\) lies to the left of every vertex of \(B_j\) in the underlying order. Color every edge within a block red and every edge between distinct blocks blue.

Suppose that \(X\) is a set of \(k\) vertices.  Then two blocks each contain at least two vertices of
\(X\). Otherwise, we have
\[
  |X|\leq b+(a-1)=a+b-1=k-1,
\]
a contradiction.   In the first block choose consecutive vertices \(x<x'\) of \(X\),
and in the second choose consecutive vertices \(y<y'\) of \(X\).  These are
also consecutive pairs in the full induced order on \(X\), since the blocks
are intervals.

The edges \(xx'\) and \(yy'\) are red, whereas all four edges between \(\{x,x'\}\) and \(\{y,y'\}\) are blue.  Then the dependency digraph contains the arc \(y\to x\). Similarly, since \(yy'\) is a red consecutive pair and \(x\) is joined
in blue to both \(y\) and \(y'\), it gives the arc \(x\to y\).  Thus \(D_\chi(X)\) contains a directed two-cycle.  No \(k\)-set is admissible in
this coloring, proving the lower bound.
\end{proof}

\section{Proof of the Main Theorem}

We need the following key lemma.

\begin{lemma}\label{lem:splicing}
Let \(K_N\) be an ordered complete graph, and let
\( \chi\) be an arbitrary red--blue edge-coloring of \(K_N\).  For each vertex
\(v\in V(K_N)\), define
\[
  L=\max_{v\in V(K_N)}\sigma(v),
  \qquad
  S_t=\{v\in V(K_N):\sigma(v)=t\}
  \quad (1\leq t\leq L).
\]
Then, for every \(t\in[L]\), there exists a set
\(X\subseteq V(K_N)\) such that
\(
  |X|=L+|S_t|-1
\)
and the coloring \(\chi\) induced on \(X\) is admissible.  In particular, if
\(
  w=\max_{1\leq t\leq L}|S_t|,
\)
then \(\chi\) contains an admissible vertex set of order \(L+w-1\).
\end{lemma}

\begin{proof}
Let \(
  S_t=\{s_1<s_2<\cdots<s_q\},
\)
and take a blue path of maximum order
\[
  P=(p_1<p_2<\cdots<p_L).
\]
By Lemma~\ref{lem:layers}(ii), \(P\) meets \(S_t\) in the unique vertex
\(
  p=p_{L-t+1}.
\)
Let
\(
  P^-=(p_1<\cdots<p_{L-t+1})
\).  Since \(\sigma(s_q)=t\), there is a blue path
of order \(t\) starting at the rightmost vertex of \(S_t\); write it as
\[
  Q^+=(q_1=s_q<q_2<\cdots<q_t).
\]
Define
\begin{equation}\label{eq:X-definition}
  X=V(P^-)\cup S_t\cup V(Q^+).
\end{equation}

We have \(V(P^-)\cap S_t=\{p\}\).  Moreover,
\(Q^+\setminus\{s_q\}\) lies strictly to the right of \(s_q\), whereas
\(P^-\cup S_t\) lies weakly to its left.  Hence
\[
  V(Q^+)\cap\bigl(V(P^-)\cup S_t\bigr)=\{s_q\},
\]
including the case \(p=s_q\).  Therefore
\begin{equation}\label{eq:X-size}
  |X|=(L-t+1)+(q-1)+(t-1)=L+q-1.
\end{equation}

It remains to prove that \(D_\chi(X)\) is acyclic.
Suppose, for a
contradiction, that \(D_\chi(X)\) contains a directed cycle \(C\).  Let
\[
  A=V(P^-)\setminus S_t,
  \qquad
  B=V(Q^+)\setminus S_t.
\]
By the intersection properties established above, we have
\(
  X=A\cup\,S_t\,\cup\,B.
\)

We observe that, whenever two consecutive vertices of \(X\) are joined by a red edge, at least one of them belongs to \(S_t\).  Indeed, let \(u<v\) be consecutive vertices
of \(X\) with \(u,v\notin S_t\).  If \(u,v\in A\), then they must be
consecutive vertices of \(P^-\), since every vertex of \(P^-\) belongs
to \(X\).  Hence \(uv\) is blue.  Similarly, if \(u,v\in B\), then
they are consecutive vertices of \(Q^+\), and again \(uv\) is blue.
A vertex of \(A\) and a vertex of \(B\) cannot be consecutive in \(X\),
because \(s_q\in S_t\) lies between them.

We next show that \(C\) cannot contain a vertex of \(S_t\).  In fact,
every vertex of \(S_t\) has out-degree zero in \(D_\chi(X)\).  Suppose
that \(z\to y\) is a dependency arc with \(z\in S_t\).  The red
consecutive pair generating this arc contains \(y\), and, by the
preceding observation, it has an endpoint \(s\in S_t\).  Since \(z\)
is the third vertex generating the arc, \(z\neq s\), and \(zs\) must
be blue.  This contradicts Lemma~\ref{lem:layers}(i), according to
which every edge within \(S_t\) is red.  Thus no vertex of \(S_t\)
has positive out-degree, and consequently \(C\cap S_t=\varnothing\).

Similarly, \(C\) cannot contain a vertex of \(B\).  Since \(s_q\) is
the rightmost vertex of \(S_t\), the terminal segment of \(X\) is
\[
  s_q=q_1<q_2<\cdots<q_t.
\]
Every edge \(q_iq_{i+1}\) is blue.  Hence no vertex of
\(B=\{q_2,\ldots,q_t\}\) is an endpoint of a red consecutive pair in
\(X\).  It follows that every vertex of \(B\) has in-degree zero in
\(D_\chi(X)\), so \(C\cap B=\varnothing\).

Therefore every vertex of \(C\) belongs to \(A\).  We now show that
every dependency arc with both endpoints in \(A\) points strictly to
the right.  Let \(z\longrightarrow y
  \)
be such an arc.  The red consecutive pair generating this arc contains
\(y\).  Since every red consecutive pair has an endpoint in \(S_t\),
this pair must be of the form \(\{y,s\}\) for some \(s\in S_t\).
Moreover, \(z\) is joined in blue to both endpoints of this pair, so
\(zs\) is blue.

By Lemma~\ref{lem:layers}(ii), every vertex of \(A\) has
\(\sigma\)-value greater than \(t\).  In particular,
\[
  \sigma(z)>t=\sigma(s).
\]
Since \(zs\) is blue, we get \(z<s\). The consecutiveness of \(y\) and \(s\) in \(X\) prevents the third selected vertex \(z\) from
lying strictly between them. Since \(z<s\), we obtain \(z<y\).
Consequently, every dependency arc with both endpoints in \(A\) points
strictly from left to right. Then we have \(C\cap A=\varnothing\), which is impossible.  Therefore the coloring induced by \(\chi\) on \(X\) is
admissible.
\end{proof}

\begin{remark}\label{rem:asymmetry}
The direct union of an arbitrary longest blue path and a largest layer need
not be admissible.  For example, on \(0<1<\cdots<6\), color exactly
\[
  02,04,05,15,23,26,35,46
\]
blue.  The unique longest blue path is \((0,2,3,5)\), and the largest layer is
\(S_2=\{1,3,4\}\).  Their union has the directed two-cycle
\(0\leftrightarrow5\).  Lemma~\ref{lem:splicing} instead keeps the prefix
\((0,2,3)\) and starts a new blue path \((4,6)\) at the rightmost layer point,
producing the admissible set \(\{0,1,2,3,4,6\}\).
\end{remark}

\begin{proof}[Proof of Theorem~\ref{thm:main}]
The lower bound follows from Proposition~\ref{prop:lower}.  For the upper bound, let
\[
  N=\floor{\frac{k^2}{4}}+1
\]
and consider an arbitrary red--blue coloring of the ordered complete graph on
\(N\) vertices.  Let \(L\) be the maximum blue-path order and let
\(w=\max_t|S_t|\).  The layers partition the vertex set, and therefore
\begin{equation}\label{eq:N-Lw}
  N=\sum_{t=1}^{L}|S_t|\leq Lw.
\end{equation}

If \(L+w\leq k\), then the integer form of the arithmetic--geometric mean
inequality gives
\[
  N\leq Lw
  \leq\max_{a+b\leq k}ab
  =\floor{\frac{k^2}{4}},
\]
contrary to the definition of \(N\).  Thus \(L+w\geq k+1\).
Lemma~\ref{lem:splicing} supplies an admissible set of order
\[
  L+w-1\geq k.
\]
By Lemma~\ref{lem:trimming}, repeated endpoint deletion produces an
admissible set of order exactly \(k\).  Hence
\[
  f(k)\leq\floor{\frac{k^2}{4}}+1,
\]
which, together with Proposition~\ref{prop:lower}, proves the theorem.
\end{proof}

Recall from~\cite{BradacLiuWuXu} that \(g(k)\) is the minimum integer
\(r\) such that every \(n\)-vertex ordered graph \(G\) with no induced
monotone \(P_{2k}\) and with positive \(K_r\)-density in
\(\overline{G}\) contains a biclique in \(\overline{G}\) whose two
parts have order \(\Omega(n/\log n)\). Brada\v{c} et al. proved that
\[
  \left\lfloor \frac{k^2}{4} \right\rfloor
  < g(k) \le f(k).
\]
Consequently, Theorem~\ref{thm:main} also determines the exact value of
\(g(k)\).

\begin{corollary}\label{cor:exact-g}
Let \(k\ge 3\) be an integer. Then we have
\[
  g(k)=f(k)
      =\left\lfloor \frac{k^2}{4} \right\rfloor+1.
\]
\end{corollary}

\section*{Acknowledgements}

 This work was supported by Science and Technology Commission of Shanghai Municipality (No. 25ZR1402474) and by the Basic Scientific Research Project of the Liaoning Provincial
Department of Education (No. 320226062).

\section*{Conflicts of Interest}
The authors declare no conflicts of interest.

\section*{Data Availability Statement}
No data was used for the research described in the paper.

\end{document}